\documentclass{amsart}
\RequirePackage{float}
\RequirePackage{graphicx}
\RequirePackage{amsmath, amsthm, amssymb, amsfonts}
\RequirePackage[colorlinks=true]{hyperref}
\hypersetup{urlcolor=blue, citecolor=blue}
\RequirePackage[nameinlink]{cleveref}

\newenvironment{prf}{\noindent {\bf Proof.} }{\endprf\par}
\def \endprf{\hfill  {\ensuremath{\square}}\medskip}
\numberwithin{equation}{section}
\newtheorem{thm}{Theorem}[section]
\newtheorem{prpstn}[thm]{Proposition}
\newtheorem{lmm}[thm]{Lemma}
\newtheorem{clm}[thm]{Claim}
\theoremstyle{remark}
\newtheorem{remark}[thm]{Remark}
\theoremstyle{definition}
\crefname{equation}{}{}
\crefname{lmm}{Lemma}{Lemma}
\crefname{thm}{Theorem}{Theorem}
\crefname{prpstn}{Proposition}{Proposition}
\crefname{clm}{Claim}{Claim}
\crefname{subpart}{Part}{Part}

\newcounter{part0}[subsection]
\renewcommand{\part}[1][]{\noindent\stepcounter{part0}{\bfseries Part \number\value{part0}:} #1.\par}
\newcounter{part1}[part0]
\newcommand{\subpart}[1][]{\noindent\stepcounter{part1}{\bfseries Part \number\value{part0}.\number\value{part1}:} #1.\par}
\newcounter{part2}[part1]
\newcommand{\subsubpart}[1][]{\noindent\stepcounter{part2}{\bfseries Part \number\value{part0}.\number\value{part1}.\number\value{part2}:} #1.\par}

\DeclareMathOperator{\supp}{supp}
\renewcommand{\d}{\, \mathop{\!\mathrm{d}}\!}
\newcommand{\dyw}{:=}
\newcommand{\R}{{\mathbb{R}}}
\newcommand{\Roma}[1]{\uppercase\expandafter{\romannumeral#1}}
\newcommand{\fl}[1]{\mathcal{#1}}
\renewcommand{\k}[3]{\mathopen{}\left#1 #2 \right#3}

\newcommand{\al}{\alpha}    
\newcommand{\de}{\delta}    
  \newcommand{\ep}{\varepsilon}

\newcommand{\pa}{\partial}

\newcommand{\beeq}{\begin{equation}}\newcommand{\eneq}{\end{equation}}

\title
[semilinear wave equations with mixed nonlinear terms]
{Global existence and lifespan for semilinear wave equations with mixed nonlinear terms}

\author{Wei Dai}
\address{School of Mathematical Sciences\\ Zhejiang University\\ Hangzhou 310027,P.R.China}
\email{daiw16@zju.edu.cn}

\author{Daoyuan Fang}
\address{School of Mathematical Sciences\\ Zhejiang University\\ Hangzhou 310027,P.R.China}
\email{dyf@zju.edu.cn}

\author{Chengbo Wang}
\address{School of Mathematical Sciences\\ Zhejiang University\\ Hangzhou 310027,P.R.China}
\email{wangcbo@zju.edu.cn}
\urladdr{http://www.math.zju.edu.cn/wang}

\date{\today}

\begin{document}

\bibliographystyle{plain} 

\begin{abstract}
Firstly, we study the equation $\square u = |u|^{q_c}+ |\partial u|^p$ with small data, where $q_c$ is the critical power of \emph{Strauss} conjecture and $p\geq q_c.$ We obtain the optimal estimate of the lifespan  $\ln\k({T_\varepsilon})\approx\varepsilon^{-q_c(q_c-1)}$ in $n=3$, and improve the lower bound of $T_\varepsilon$ from $\exp\k({c\varepsilon^{-(q_c-1)}})$ to $\exp\k({c\varepsilon^{-(q_c-1)^2/2}})$  in $n=2$. Then, we study the \emph{Cauchy} problem with small initial data for a system of semilinear wave equations $\square u = |v|^q,$ $ \square v = |\partial_t u|^p$ in 3-dimensional space with $q<2$. We obtain that this system admits a global solution above a $p-q$ curve for spherically symmetric data. On the contrary, we get a new region where the solution will blow up.
\end{abstract}

\keywords{Strauss conjecture; Glassey conjecture; generalized Strichartz estimate; Klainerman-Sobolev inequalities}

\subjclass[2010]{35L05, 35L15, 35L70}

\maketitle

\section{Introduction}
In this paper, we want to study the global solvability and the blow up for some semilinear wave equations with nonlinear terms like $|u|^q$ and $|\partial_t u|^p$. Firstly we study the lifespan of the equation with mixed nonlinear terms
\begin{equation}\label{e1.1}
\begin{cases}
\square u\dyw \partial_t^2u-\Delta u=|\partial_t u|^p+|u|^{q}, \\
(u, \partial_tu)|_{t=0}=\varepsilon\big(f(x), g(x)\big),
\end{cases}
\end{equation}
where $p>1,$ $q>1$ and $x\in \R^n.$ 
This equation is in relation (\cite{MR3552253}) with the following equations which are well-investigated:
\begin{gather}
\square v=|v|^q,~t>0,~x\in\R^n,\label{e1.2}\\
\square w=|\partial_t w|^p,~t>0,~x\in\R^n.\label{e1.3}
\end{gather}
The first equation \cref{e1.2} is related to the \emph{Strauss} conjecture, for which the critical power, denoted by $q_c(n)$, is known to be the positive root of the quadratic equation
\begin{equation*}
(n-1)q^2-(n+1)q-2=0\ .
\end{equation*}
This conjecture was finally verified in \cite{MR1481816,MR2195336}. And the complete history can be found in \cite{MR3013062}. As for the other equation \cref{e1.3}, which is related to the \emph{Glassey} conjecture, see \cite{MR2980460} and the references therein for more information.

The global existence and blow up dichotomy for the equation \cref{e1.1}  with spatial dimension $n=2, 3$ has been well understood, through the works \cite{MR3552253, Wei2014Blow}.
For the cases where there is no global existence, it is also interesting to give sharp estimates of the lifespan, $T_\ep$, from above and below. 
In  \cite{MR3552253}, the sharp estimates of $T_\ep$ has been obtained for 
any $p, q\ge 2$ and $q>2/(n-1)$, except for the critical case $p\geq q=q_c$. More precisely, it is known that 
\begin{equation*}
\exp\k({c\varepsilon^{-(q_c-1)}})\leq T_\varepsilon\leq\exp\k({C\varepsilon^{-q_c(q_c-1)}}), \  p\geq q=q_c,\ n=2, 3.
\end{equation*}
The lower bound was obtained in \cite{MR3552253} by using a variant of \emph{Klainerman's} method of commuting vector fields, and the upper bound comes from the discussion of \emph{Strauss} conjecture and a simple application of comparison principle, which is expected to be sharp for this problem.

It is not difficult to find that  the lower bound of the lifespan of the critical problem of \cref{e1.2} 
 is closely related to the power $q$ of $L_t^q$ in time norm for the forcing term in the key estimates. For example, the obtained bound $\exp\k({c\varepsilon^{-q_c(q_c-1)}})$ which comes from \cite{MR1408499}  coincides with $q=q_c$ in estimates \cref{e3.1}, and the obtained bound $\exp\k({c\varepsilon^{-(q_c-1)^2/2}})$ which comes from \cite{MR2911108}  coincides with $q=(q_c-1)/2$ in estimates \cref{e3.2}. In order to improve the result from \cite{MR3552253}, we adapt these generalized \emph{Strichartz} estimates to the equation \cref{e1.1}, use energy inequality with \emph{Klainerman-Sobolev} inequality to deal with derivative term. Thus we get the following result
for dimension three, which is sharp in general.
\begin{thm}\label{t1.1}
Let $n=3$, $q=q_c(3)=1+\sqrt{2}$ and $p\geq q$. Suppose that the data $(f, g)$ satisfy
\begin{equation}\label{e1.4}
\Lambda \dyw   \|\k<x>^5\partial_x^{\leq 3}f\|_{L_x^\infty}+\|\k<x>^5\partial_x^{\leq 2}g\|_{L_x^\infty}<\infty,
\end{equation}
then there exists an $\varepsilon_0(\Lambda,  p)>0$ and a constant $c>0$, such that for any $\varepsilon\in(0, \varepsilon_0), $ \cref{e1.1} has a unique solution $u\in C^0\k({\k[{0, \bar T_\varepsilon}];H^3(\R^3)})\cap C^1\k({\k[{0, \bar T_\varepsilon}];H^2(\R^3)})$ where
 \begin{equation*}
\bar T_\varepsilon\dyw
\exp\k({c\varepsilon^{-q_c(q_c-1)}}).
\end{equation*}
\end{thm}
Turning to the $2$ dimensional case,
as is typical for wave equations, the problem seems to be more delicate.
As far as the authors are aware, even for the
problem \cref{e1.2} with $q=q_c(2)$,
the only available approach to prove the sharp lower bound is given in
\cite{MR1233659}, for compactly supported smooth data, which relies heavily on the fundamental solutions and seems inappropriate for the problems involving nonlinear term $|\pa_t u|^p$.
Instead, the approach using space-time estimates is more robust for various nonlinear problems.
Here, we adapt the generalized \emph{Strichartz} estimates 
from \cite{MR2911108} to the current setting and obtain the following
\begin{thm}\label{t1.2} 
Let $n=2$, $q=q_c(2)=(3+\sqrt{17})/2$ and $p\geq q$. Considering
\cref{e1.1} with data $(f, g)$ satisfying \eqref{e1.4},
we have a similar existence result with small data, as in Theorem \ref{t1.1}, with $\bar T_\varepsilon$ replaced by
 \begin{equation*}
\tilde T_\varepsilon\dyw\exp\k({c\varepsilon^{-(q_c-1)^2/2}}).
\end{equation*}
\end{thm}


\par Next, we consider a coupled wave system with different nonlinear terms in each equation,
\begin{equation}\label{e1.5}
\begin{cases}
\square u=|v|^q, \qquad\square v=|\partial_tu|^p,\\
(u, \partial_tu)|_{t=0}=\varepsilon\big(f(x), g(x)\big),~(v, \partial_tv)|_{t=0}=\varepsilon\big(\tilde f(x), \tilde g(x)\big),
\end{cases}
\end{equation}
where $u, v$ depend on $(t, x)\in[0, T)\times\R^3$ for some $T\in(0, \infty],$ $\varepsilon$ is positive and small enough.

 The system \cref{e1.5} has been discussed in \cite{MR3485160}, which shows there exists a curve in $(1, \infty)^2$ of index-pairs $(p, q),$
\begin{equation}\label{e1.6}
(p-1)(pq-1)=p+2, \qquad 1<q, \qquad 1<p<3.
\end{equation}
When $(p, q)$ lies below the curve, this system blow up in most cases whatever small $\varepsilon$ is. On the contrary, when $(p, q)$ lies above the curve with $2<p<3,$ $2<q,$ this system \cref{e1.5} has a global solution at least for radially symmetric small data.
\par To analyze this equation, we compare it with some closely related systems. They are
\begin{equation*}
\begin{aligned}
\text{(\Roma1)}:\qquad  &\square u=|v|^q, \qquad&&\square v=|u|^p, \qquad&&x\in \R^3;\\
\text{(\Roma2)}:\qquad  &\square u=|\partial_tv|^q, \qquad&&\square v=|\partial_tu|^p, \qquad&&x\in \R^3.
\end{aligned}
\end{equation*}
\par For (\Roma1), which relates to the \emph{Strauss} conjecture, it is known that
\begin{equation}\label{e1.7}
\max\k\{{\frac{p+2+q^{-1}}{pq-1}, \frac{q+2+p^{-1}}{pq-1}}\}=1
\end{equation}
is the critical curve of index-pairs $(p, q)$. The curve was firstly provided in \cite{MR2033494}, in which they prove global existence for supercritical case and blow up for subcritical case. For blow up in the critical case and the lifespan estimates, see \cite{MR1785116} and the references therein.

As for (\Roma2), which relates to the \emph{Glassey} conjecture, such a curve is
 \begin{equation}\label{e1.8} 
\max\k\{{\frac{q+1}{pq-1}, \frac{p+1}{pq-1}}\}=1.
\end{equation}
This is optimal at least for radially symmetric initial data, where the blow up part can be found in \cite{MR1733070} and the existence part for symmetric situation was verified in \cite{MR2221121}. We refer the interested readers to \cite{MR3485160} for more information about these two problems.
\par It is naturally to infer that the critical curve for \cref{e1.5} should lies between the curves \cref{e1.7} and \cref{e1.8}. However, the curve \cref{e1.6} intersects with one of the above curves (see the figure below, $CC''$ intersects with $l_B$). This motivates us to improve the result when $q<2.$

\begin{figure}[H]
\centering
\includegraphics[width=0.7\textwidth]{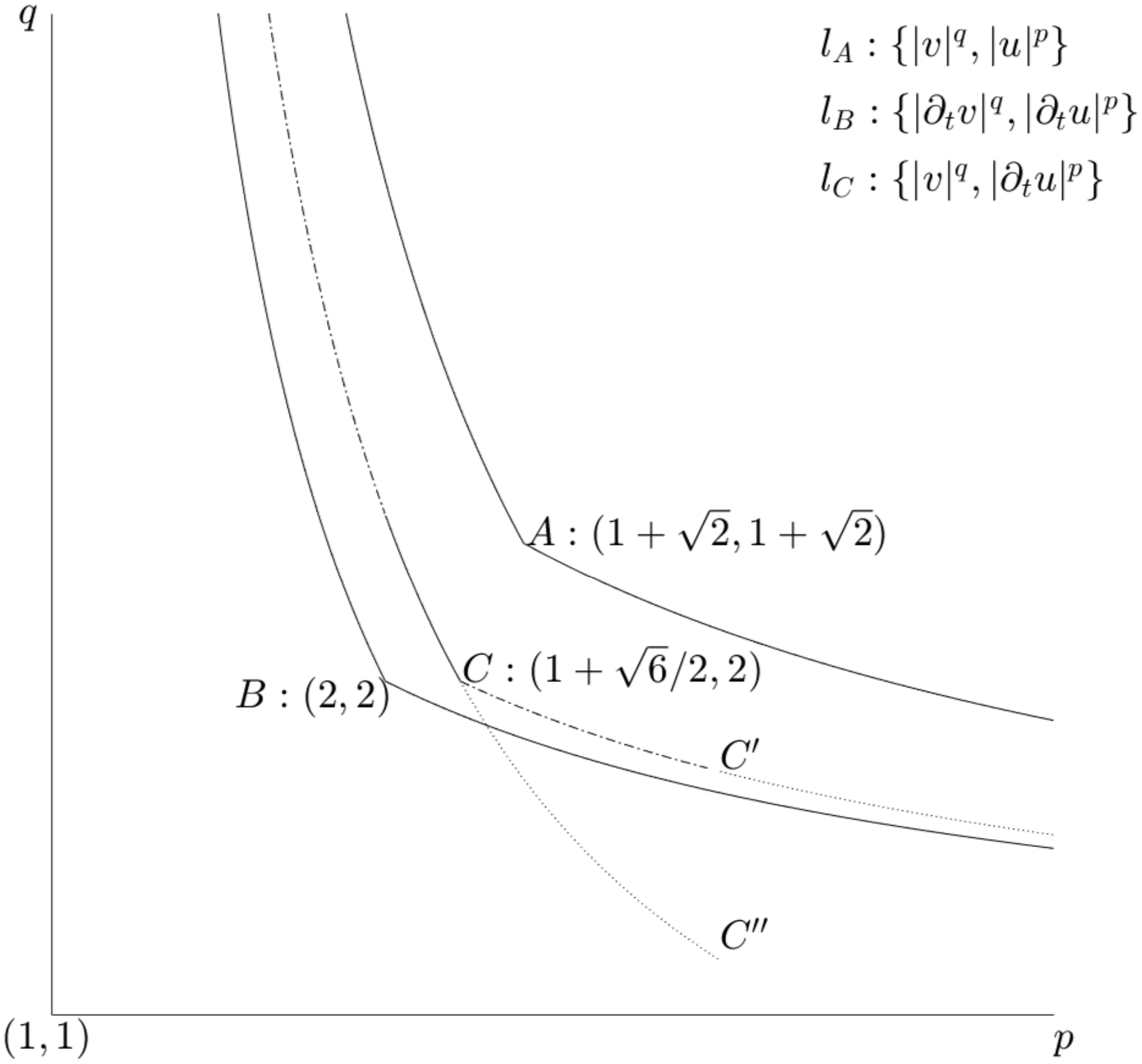}
\end{figure}
Here $l_A$ is the critical curve to problem (\Roma1), $l_B$ is the critical curve to problem (\Roma2) and $CC''$ is \cref{e1.6} with $q<2.$
We want to establish a global existence theorem to \cref{e1.5} with radially symmetric small data, for the region above $CC',$
\begin{equation}\label{e1.9}
p(q-1)>2+\frac{1}{pq}, \qquad 1<q<2, \qquad 2<p<3.
\end{equation}

\begin{thm}\label{t1.3}
Suppose that \cref{e1.9} is satisfied. We also suppose $f, \tilde f\in C^2(\R^3),$ $g, \tilde g\in C^1(\R^3)$ are spherically symmetric and supported in $B_1(0).$ Then there exists a positive number $\varepsilon_0$ such that for $0 <\varepsilon<\varepsilon_0,$ there exists a global solution $(u, v)$ of \cref{e1.5} satisfying
\begin{equation*}
u, v, \partial_t u\in C(\R^+\times \R^3).
\end{equation*}
\end{thm}

The key point here is that a symmetric $3$-D wave equation is equivalent to an $1$-D equation essentially, in which case the solution of linear problem has a higher regularity. In order to match the situation $q<2$, we use the weight functions different from \cite{MR3485160}.

As for the blow up part, we adopt the strategy of deriving a system of ordinary differential  inequalities which causes blow up solutions. Since the technique is suitable for any dimensions, we give a general result rather than $n=3$.

\begin{thm}\label{t1.4}
Let $n\ge 1$ and $p, q>1$ such that \begin{equation}\label{e1.10}
\frac{p+1}{pq-1}>\frac{n-1}{2}, \qquad 
\frac{n-1}{2}(q-1)<1
\ .
\end{equation}
Suppose that all data are supported in $B_1(0)$ and $g-f, f, \tilde g-\tilde f, \tilde f$ are non-negative where $\tilde f$ does not vanish identically. Then for any $\varepsilon>0,$ there are no global weak solutions 
of \cref{e1.5} such that  $\supp (u,v)\subset\{(t, x):|x|\leq t+1\}$, and
\begin{equation}\label{eq-1.11}
u, v, \partial_t u, \partial_t v\in C (\R_+;L^1(\R^n) ), 
v\in C(\R_+; L^q(\R^n)),  \partial_t u\in C(\R_+;L^p(\R^n)),
\end{equation}
where $\R_+=[0,\infty)$.
\end{thm}
 This improves the result from the curve $CC''$ to the right branch of $l_B$. Remark also that this result is better than that in \cite{aaa}.
 
\begin{remark} For $n=3$,
partly because the endpoint $(p,q)=(\sqrt{6}/2+1, 2)$ is critical,
but $(p,q)=(2,2)$ on $l_B$ is not,
 we tend to believe that the critical curve 
for \cref{e1.5} with $q\in (1,2)$ is
that appeared in Theorem \ref{t1.3}, that is,
$$p(q-1)=2+\frac{1}{pq}, \qquad 1<q<2, \qquad 1<p $$
which is $CC'$ in the figure. 
\end{remark}

This rest of paper is organized as follows. In Section \ref{s2} we introduce some notations we will use. Then we prove Theorem \ref{t1.1}-\ref{t1.4}, in Section \ref{s3}-\ref{s5}, respectively.

\section{notations}\label{s2}
We list here some notations which will be used in our article.
First, the \emph{Einstein} summation convention is used, as well as the convention
that Greek indices $\mu, \nu, \cdots$ range from $0$ to $n$ while Latin indices $i, j, \cdots$ will run from $1$ to $n.$
\par We also denote
\begin{alignat*}{2}
&\k<a>\dyw \sqrt{1+|a|^2}\approx1+|a|, \\
&\|f(x)\|_{\fl{L}_r^p L_\omega^b}\dyw\k({\int_0^\infty \k({\int_{S^{n-1}}|f(r\omega)|^b\d \omega})^{p/b}r^{n-1}\d r})^{1/p}, \\
&\partial u(t, x)\dyw \{\partial_\mu u\}=(\partial_t u, \partial_x u), \\
&\partial ^{\leq k}f\dyw\{\partial^{\alpha}f\}_{0\leq|\alpha|\leq k.}
\end{alignat*}

Furthermore, we will use some kinds of special vector fields. There are spatial rotation: $\Omega_{jk}=x^k\partial_j-x^j\partial_k,$ Lorentz boost: $\Omega_{0k}=t\partial_k+x^k\partial_t$ and scaling: $L_0=t\partial_t+x^i\partial_i.$ Set $\Gamma=\{\partial_\mu, \Omega_{\mu\nu}, L_0\}$ be the well-known \emph{Klainerman} vector fields. For such vector fields, we have
\begin{equation}\label{e2.1}
\begin{aligned}
&\k[{\partial_\mu, \Gamma_\alpha}]f=C_{\mu\alpha}^\beta \partial_\beta f, \qquad\k[{\Gamma_\alpha, \Gamma_\beta}]f=C_{\alpha\beta}^\sigma \Gamma_\sigma f, \\
&\k[{\Gamma_\alpha, \square}]f=C_\alpha\, \square f, \qquad\Big(\alpha, \beta=0, 1, \cdots, (n^2+3n+2)/2\Big)
\end{aligned}
\end{equation}
\noindent where $[X, Y]$ denotes the commutator $XY-YX.$  And all coefficients are belong to $C_b^\infty.$
\par We denote a constant $C$ which may change from line to line, but not depend on $\varepsilon, t$ or $x.$ And $A\lesssim B$ means $A\leq CB$ for some $C>0,$ $A\gtrsim B$ is similar, $A\approx B$ means $A\lesssim B\lesssim A.$

\section{Proof of Theorem \ref{t1.1} and Theorem \ref{t1.2}}\label{s3}

\subsection{Preliminaries}\

We firstly list some lemmas to be used later.
\begin{lmm}[Local existence]\label{l3.1}
When $n\leq 3.$ For the equation
\begin{equation*}
\begin{cases}
\partial_t^2 u-\Delta u=F(u, \partial u), \\
u(0, x)=f(x), u_t(0, x)=g(x).
\end{cases}
\end{equation*}
If $(f, g)\in H^3\times H^2$, $F\in C^2$ and $F(0, 0)=0$ then there is a $T>0, $ depending on the norm of the data, so that this Cauchy problem has a unique solution satisfying
\begin{equation*}
\|\partial^{\leq 3}u(t, \cdot)\|_{L_x^2}<\infty, \quad 0\leq t\leq T.
\end{equation*}
Also, if $T_*$ is the supremum over all such times $T,$ then either $T_*=\infty$ or
\begin{equation*}
|\partial^{\leq 3}u|\not\in L_t^\infty L_x^2(0\leq t<T_*).
\end{equation*}
\end{lmm}
\begin{proof}
The result is classical, for a proof, see, e.g., Chapter 12 of \cite{MR2597943}.
\end{proof}

\begin{lmm}[Theorem 6.4 of \cite{MR1408499}]\label{l3.2}
When $n=3,$ suppose $u$ solves the linear equation
\begin{equation*}\begin{cases}
\square u(t, x)=F(t, x), \\
u(0, x)=f(x), u_t(0, x)=g(x)
\end{cases}\end{equation*}
in $[0, T_*)\times\R^3.$ Then there exists a constant $C,$ such that $\forall\, T<T_*$
\begin{equation}\label{e3.1}
\begin{aligned}
T^{1/q_c^2}\|u(T, \cdot)\|_{\fl{L}_r^{q_c}L_\omega^3}\leq& C\Big( \|r^{1-2/q_c}\partial_\omega^{\leq 1}f\|_{\fl{L}_r^{q_c}L_\omega^{3/2}}+ \|r^{1+1/q_c}\partial_\omega^{\leq 1}f\|_{\fl{L}_r^\infty L_\omega^{3/2}}\\
&+ \|r^{2-2/q_c}g\|_{\fl{L}_r^{q_c}L_\omega^{3/2}}+ \|r^{2+1/q_c}g\|_{\fl{L}_r^\infty L_\omega^{3/2}}\\
&+ \|F\|_{L_t^{q_c}\fl{L}_r^1 L_\omega^{3/2}(t<T/4)}+T^{1/q_c}\|F\|_{L_t^{\infty}\fl{L}_r^1 L_\omega^{3/2}(T/4<t<T)}\Big).
\end{aligned}
\end{equation}
\end{lmm}

\begin{lmm}[Theorem 1.4 of \cite{MR2911108}]\label{l3.3}
When $n=2,$ suppose $u$ solves the linear equation
\begin{equation*}\begin{cases}
\square u(t, x)=F(t, x), \\
u(0, x)=f(x), u_t(0, x)=g(x)
\end{cases}\end{equation*}
in $[0, T_*)\times\R^2.$ Let $q_*=\frac{q_c-1}{2}$, $s_c=1-1/q, b>\frac{1}{q_*q_c}+\frac{1}{q_c}$, $X^b\equiv H_\omega^{s_c+\delta, b}\times H_\omega^{s_c-1+\delta, b}(\delta>0),$ where
$$H_\omega^{s, b}=\{u\in H^s:\k\|{(1-\Delta_\omega)^{b/2}u}\|_{H^s}<\infty\}, \qquad \Delta_\omega=\sum_{1\leq i<j\leq n}\Omega_{ij}^2.$$
Then there exists a constant $C,$ such that $\forall\, T<T_*$
\begin{equation}\label{e3.2}
\|u\|_{L_t^{q_*q_c}\fl{L}_r^{q_c} L_\omega^2(t<T)}\leq C({\ln(2+T)})^{1/(q_*q_c)}\Big(  \|(f, g)\|_{X^b}+\|F\|_{L_t^{q_*}\fl{L}_r^{1} L_\omega^2}\Big).
\end{equation}
\end{lmm}

\begin{lmm}[Energy inequality]\label{l3.4}
For any $n,$ suppose $u$ solves the linear equation
\begin{equation*}\begin{cases}
\square u(t, x)=F(t, x), \\u(0, x)=f(x), u_t(0, x)=g(x).
\end{cases}\end{equation*}
We have, $\forall\, T>0$
\begin{equation*}
\|\partial u(T, \cdot)\|_{L_x^2}\leq\|\partial u(0, \cdot)\|_{L_x^2}+\|F\|_{L_t^1L_x^2(t<T)}
\end{equation*}
Since the property of commutator \cref{e2.1}, we also have for any $k\in \mathbb{Z}^+$
\begin{equation*}
\|\partial \Gamma^{\leq k} u(T, \cdot)\|_{L_x^2}\leq C_k\k({\|\partial \Gamma^{\leq k}u(0, \cdot)\|_{L_x^2}+\|\Gamma^{\leq k}F\|_{L_t^1L_x^2(t<T)}}).
\end{equation*}
\end{lmm}

\begin{lmm}[\emph{Klainerman-Sobolev}  inequalities]\label{l3.5}
Let $v\in C^\infty(\R^{1+n})$ vanish when $|x|$ is large, $1\leq p<\infty$ and $k>n/p,$ then if $\ t>0$
\begin{equation}\label{e3.3}
\k<{t+r}>^{(n-1)/p}\k<{t-r}>^{1/p}|v(t, x)|\leq C\|\Gamma^{\leq k} v(t, \cdot)\|_{L_x^p}.
\end{equation}
\par Otherwise if $\ 1\leq p<q<\infty$ and $k\geq n/p-n/q,$ then
\begin{equation}\label{e3.4}
\k<t>^{(n-1)(1/p-1/q)}\|v(t, \cdot)\|_{L_x^q}\leq C\|\Gamma^{\leq k} v(t, \cdot)\|_{L_x^p}.
\end{equation}
\end{lmm}
\begin{proof}
The inequality \cref{e3.3} is the well-known \emph{Klainerman-Sobolev} inequality, which is proved in \cite{MR865359}. Heuristically, the inequality \cref{e3.4}                     can be viewed as a consequence of \cref{e3.3}. We refer \cite[(3.9)]{MR1386769} for a complete proof.
\end{proof}

\begin{lmm}[Sobolev embedding]\label{l3.6}
Let $x\in \R^n,$ and $1\leq p<q\leq \infty,$ then
\begin{equation}\label{e3.5}
\|f(x)\|_{L_x^q}\leq C \|\partial^{\leq k}f(x) \|_{L_x^p},
\end{equation}
where $k\geq n/p-n/q$ if $\ q<\infty,$ or $k> n/p$ if $\ q=\infty.$ When it comes to $S^{n-1},$ with $1\leq p<q\leq \infty,$ then
\begin{equation}\label{e3.6}
\|f(\omega)\|_{L_\omega^p}\leq C \|\Omega^{\leq k}f(\omega) \|_{L_\omega^p},
\end{equation}
where $k\geq (n-1)/p-(n-1)/q$ if $\ q<\infty,$ or $k> (n-1)/p-(n-1)/q$ if $\ q=\infty.$ For mixed-norm, we have
\begin{equation}\label{e3.7}
\|f(x)\|_{\fl{L}_r^q L_\omega^4}\leq C \|\partial^{\leq k}f(x) \|_{\fl{L}_r^p L_\omega^2},
\end{equation}
where $2\leq p\leq q \leq 4$ and $k\geq n/4.$
\end{lmm}
\begin{prf}
Here inequality \cref{e3.5} is known as \emph{Sobolev} embedding, and inequality \cref{e3.6} just the same inequality on sphere. Finally the inequality \cref{e3.7} is a simple Corollary of lemma 3.2 in \cite{MR3724252} and normal \emph{Sobolev} embedding in a small ball contains origin point.
\end{prf}

\begin{clm}\label{c3.7}
We claim that there exists a constant $C_0$ depends on $\Lambda, n$ but not  on $\varepsilon$ if $\varepsilon$ is small enough, then all of the initial norms we will use can be bounded by $C_0\varepsilon$ .
\end{clm}
The proof of Claim \ref{c3.7} need some delicate calculation, so we will postpone it to the end of this section. For now, we are ready to prove the Theorem \ref{t1.1} and Theorem \ref{t1.2}.
\subsection{Proof of Theorem \ref{t1.1}(n=3)}\ \par
By Lemma \ref{l3.1}, we only need to prove that if $T_\varepsilon\leq\bar T_\varepsilon,$ then $\|\partial^{\leq 3}u\|_{L_t^\infty L_x^2(t<T_\varepsilon)}\leq M<\infty$, which gives the contradiction.
\par Firstly, we want to prove that for some suitable $C_M, c, \varepsilon_0$ to be fixed later, if $\ 0\leq T<T_\varepsilon\leq\bar T_\varepsilon$ then
\begin{equation}\label{e3.8}
\begin{aligned}
&A(T)=\k<T>^{1/q_c^2}\|\Gamma^{\leq 2}u(T, \cdot)\|_{\fl{L}_r^{q_c}L_\omega^3}&\leq C_M \varepsilon, \\
&B(T)=\|\partial\Gamma^{\leq 2}u(T, \cdot)\|_{L_x^2}&\leq C_M \varepsilon.
\end{aligned}
\end{equation}

With the help of Claim \ref{c3.7}, we choose $C_M$ satisfying that $A(0), B(0)\leq C_M\varepsilon/4.$ Take $X=\{T\in [0, T_\varepsilon):A(t), B(t)\leq C_M\varepsilon, \forall\, t\in[0, T]\}.$ We want to prove $X=[0, T_\varepsilon).$ Consequently, it would suffice to show that,
if $T$ is as above, then equation \cref{e3.8} implies that $A(T), B(T)\leq C_M\varepsilon/2$. Because of Lemma \ref{l3.1}, we assume $T>1$ without loss of generality.

\part[Estimate of $A(t)$]
By Lemma \ref{l3.2}, we obtain
\begin{equation}\label{e3.9}
\begin{aligned}
A(T)&\leq C_M\varepsilon/4+C\k\|{\Gamma^{\leq 2}|u|^{q_c}}\|_{L_t^{q_c}\fl{L}_r^1 L_\omega^{3/2}(t<T/4)}\\
&\quad+C\k\|{\Gamma^{\leq 2}|\partial u|^p}\|_{L_t^{q_c}\fl{L}_r^1 L_\omega^{3/2}(t<T/4)}\\
&\quad+CT^{1/q_c}\k\|{\Gamma^{\leq 2}|u|^{q_c}}\|_{L_t^{\infty}\fl{L}_r^1 L_\omega^{3/2}(T/4<t<T)}\\
&\quad+CT^{1/q_c}\k\|{\Gamma^{\leq 2}|\partial u|^p}\|_{L_t^{\infty}\fl{L}_r^1 L_\omega^{3/2}(T/4<t<T)}.
\end{aligned}
\end{equation}
\par First we consider the spatial norm part of $\ |u|$ terms, by \emph{H\"older's} inequality and Lemma \ref{l3.6}, we see
\begin{equation*}
\begin{aligned}
\k\|{\Gamma^{\leq 2}|u|^{q_c}}\|_{\fl{L}_r^1 L_\omega^{3/2}}&\leq C\k\|{|\Gamma^{\leq 2}u||u|^{q_c-1}}\|_{\fl{L}_r^1 L_\omega^{3/2}}+C\k\|{|\Gamma^{\leq 1}u|^2|u|^{q_c-2}}\|_{\fl{L}_r^1 L_\omega^{3/2}}\\
&\leq C\k\|{\Gamma^{\leq 2}u}\|_{\fl{L}_r^{q_c}L_\omega^3}\k\|u\|_{\fl{L}_r^{q_c} L_\omega^{3\sqrt2}}^{q_c-1}+C\k\|{\Gamma^{\leq 1}u}\|_{\fl{L}_r^{q_c}L_\omega^3}^2\k\|{u}\|_{\fl{L}_r^{q_c} L_\omega^\infty}^{q_c-2}\\
&\leq C\k\|{\Gamma^{\leq 2}u}\|_{\fl{L}_r^{q_c}L_\omega^3}^{q_c}\\
&\leq C\k({\k<t>^{-1/q_c^2}A(t)})^{q_c}\\
&\leq CC_M^{q_c}\varepsilon^{q_c}\k<t>^{-1/q_c}.
\end{aligned}
\end{equation*}
\par Then, we consider the spatial norm part of $\ |\partial u|$ terms, by \emph{H\"older's} inequality and Lemma \ref{l3.5}, Lemma \ref{l3.6}, we have
\begin{equation*}
\begin{aligned}
\k\|{\Gamma^{\leq 2}|\partial u|^{p}}\|_{\fl{L}_r^1 L_\omega^{3/2}}&\leq C\k\|{|\Gamma^{\leq 2}\partial u||\partial u|^{p-1}}\|_{\fl{L}_r^1 L_\omega^{3/2}}+C\k\|{|\Gamma^{\leq 1}\partial u|^2|\partial u|^{p-2}}\|_{\fl{L}_r^1 L_\omega^{3/2}}\\
&\leq C\k\|{\Gamma^{\leq 2}\partial u}\|_{\fl{L}_r^2 L_\omega^2}\k\|{\partial u}\|_{\fl{L}_r^2 L_\omega^6}\k\|{\partial u}\|_{\fl{L}_r^\infty L_\omega^\infty}^{p-2}+C\k\|{\Gamma^{\leq 1}\partial u}\|_{\fl{L}_r^2 L_\omega^3}^2\k\|{\partial u}\|_{\fl{L}_r^\infty L_\omega^\infty}^{p-2}\\
&\leq C\k\|{\Gamma^{\leq 2}\partial u}\|_{L_x^2}^2\k\|{\partial u}\|_{L_x^\infty}^{p-2} \\
&\leq C\k\|{\Gamma^{\leq 2}\partial u}\|_{L_x^2}^p\k({\k<t>^{-1}})^{p-2}\\
&\leq CC_M^{p}\varepsilon^{p}\k<t>^{2-p}.
\end{aligned}
\end{equation*}
\par Back to equation \cref{e3.9}, assume $\varepsilon_0<1/C_M.$ Since $\ p\geq q_c=1+\sqrt2$, $0<\varepsilon<\varepsilon_0,$ we conclude
\begin{equation*}
\begin{aligned}
A(T)&\leq C_M\varepsilon/4 +C C_M^{q_c}\varepsilon^{q_c} \k({1+\k({\int_0^{T/4}\k<t>^{-1}\d t})^{1/q_c}})\\
&\quad+C C_M^{p}\varepsilon^{p} \k({\k<T>^{q_c-p}+\k({\int_0^{T/4}\k<t>^{-q_c(p-2)}\d t})^{1/q_c}})\\
&\leq C_M\varepsilon \k({1/4+ C C_M^{q_c-1}\varepsilon^{q_c-1}({\ln(T+2)})\ ^{1/q_c}+C C_M^{p-1}\varepsilon^{p-1}})\\
&\leq C_M\varepsilon \k({1/3+ C C_M^{q_c-1}c^{1/q_c}+C C_M^{q_c-1}\varepsilon_0^{q_c-1}}).
\end{aligned}
\end{equation*}
\part[Estimate of $B(t)$]
For this part, by the Lemma \ref{l3.4}, we know that
\begin{equation}\label{e3.10}
B(T)\leq C_M \varepsilon/4 +C\k\|{\Gamma^{\leq 2}|u|^{q_c}}\|_{L_t^1L_x^2 (t<T)}+C\k\|{\Gamma^{\leq 2}|\partial u|^{p}}\|_{L_t^1L_x^2(t<T)}.
\end{equation}
\par For the $|u|$ term, 
by \emph{H\"older's} inequality and Lemma \ref{l3.5},
we have
\begin{equation*}
\begin{aligned}
\k\|{\Gamma^{\leq 2}|u|^{q_c}}\|_{L_x^2}&\leq C\k\|{\Gamma^{\leq 2}u}\|_{L_x^{q_c}}\k\|u\|_{L_x^{a}}^{q_c-1}+C\k\|{\Gamma^{\leq 1}u}\|_{L_x^4}^2\k\|u\|_{L_x^{\infty}}^{q_c-2}\\
&\leq C\k\|{\Gamma^{\leq 2}u}\|_{L_x^{q_c}}^{q_c}\k({\k<t>^{-2\k({1/q_c-1/a})(q_c-1)}+\k<t>^{-2\k({1/q_c-1/4})\cdot2-2\k({1/q_c})(q_c-2)}})\\
&\leq C\k\|{\Gamma^{\leq 2}u}\|_{L_x^{q_c}}^{q_c}\k<t>^{-1}\\
&\leq C\k({\k<t>^{-1/q_c^2}A(t)})^{q_c}\k<t>^{-1}\\
&\leq CC_M^{q_c}\varepsilon^{q_c}\k<t>^{1-q_c},
\end{aligned}
\end{equation*}
where $a=(q_c-1)/(1/2-1/q_c)$.
\par For the $|\partial u|$ term, by \emph{H\"older's} inequality and Lemma \ref{l3.5} again, we know
\begin{equation*}
\begin{aligned}
\k\|{\Gamma^{\leq 2}|\partial u|^{p}}\|_{L_x^2}&\leq C\k\|{\Gamma^{\leq 2}\partial u}\|_{L_x^2}\|\partial u\|_{L_x^{\infty}}^{p-1}+C\k\|{\Gamma^{\leq 1}\partial u}\|_{L_x^4}^2\|\partial u\|_{L_x^{\infty}}^{p-2}\\
&\leq  C\k\|{\Gamma^{\leq 2}\partial u}\|_{L_x^2}^p\k({\k<t>^{-2(1/2)(p-1)}+\k<t>^{-2(1/2-1/4)2-2(1/2)(p-2)}})\\
&\leq CC_M^{p}\varepsilon^{p}\k<t>^{1-p}.
\end{aligned}
\end{equation*}
\par Back to \cref{e3.10}, we obtain
\begin{equation*}
\begin{aligned}
B(T)&\leq C_M\varepsilon \k({1/4+ C C_M^{q_c-1}\varepsilon^{q_c-1}+C C_M^{p-1}\varepsilon^{p-1}})\leq C_M\varepsilon\k({1/4+ C C_M^{q_c-1}\varepsilon_0^{q_c-1}}).
  \end{aligned}
\end{equation*}
\part[The boundness of $A(t),B(t)$]
We choose $\varepsilon_0$ and the constant $c$ in $\bar T_\varepsilon$ small enough, such that $A(T), B(T)\leq C_M\varepsilon/2,$ which completes the proof.

\subsection{Proof of Theorem \ref{t1.2}(n=2)}\ \par
Similar to last subsection, we need to prove that for some suitable $C_M, c, \varepsilon_0$ to be fixed hereafter, if $\ 0\leq T<T_\varepsilon\leq\tilde T_\varepsilon$ then
\begin{equation}\label{e3.11}
\begin{aligned}
&A(T)=\|\Gamma^{\leq 2}u\|_{L_t^{q_*q_c}\fl{L}_r^{q_c} L_\omega^2 (t<T)}&&\leq C_M \varepsilon^{1/q_c}, \\
&B(T)=\|\partial\Gamma^{\leq 2}u(T, \cdot)\|_{L_x^2}&&\leq C_M \varepsilon,
\end{aligned}
\end{equation}
where $q_*=(q_c-1)/2$ as in Lemma \ref{l3.3}. By Claim \ref{c3.7}, we fix $C_M$ such that $B(0)\leq C_M\varepsilon/4.$ We need to show that equation \cref{e3.11} implies $A(T)\leq C_M \varepsilon^{1/q_c}/2$ and $B(T) \leq C_M \varepsilon/2$ for above $T$ and $T_\varepsilon.$

\part[Estimate of $A(t)$]
By Lemma \ref{l3.3}, and $T<\tilde T_\varepsilon,$ we get
\begin{equation}\label{e3.12}
\begin{aligned}
A(T)&\leq cC_M\varepsilon/4^{1/q_c}+cC\varepsilon^{1/q_c-1}\k\|{\Gamma^{\leq 2}|u|^{q_c}}\|_{L_t^{q_*}\fl{L}_r^{1} L_\omega^2(t<T)}\\
&\quad+cC\varepsilon^{1/q_c-1}\k\|{\Gamma^{\leq 2}|\partial u|^p}\|_{L_t^{q_*}\fl{L}_r^{1} L_\omega^2(t<T)}.
\end{aligned}
\end{equation}
For the second term, similar to $n=3$, we have
\begin{equation*}
\begin{aligned}
\k\|{\Gamma^{\leq 2}|u|^{q_c}}\|_{L_t^{q_*}\fl{L}_r^{1} L_\omega^2}&\leq C\k\|{|\Gamma^{\leq 2}u||u|^{q_c-1}}\|_{L_t^{q_*}\fl{L}_r^{1} L_\omega^2}+C\k\|{|\Gamma^{\leq 1}u|^2|u|^{q_c-2}}\|_{L_t^{q_*}\fl{L}_r^{1} L_\omega^2}\\
&\leq C\k\|{\Gamma^{\leq 2}u}\|_{L_t^{q_*q_c}\fl{L}_r^{q_c} L_\omega^2}^{q_c}\\
&\leq CC_M^{q_c}\varepsilon,
\end{aligned}
\end{equation*}
and
\begin{equation*}
\begin{aligned}
\k\|{\Gamma^{\leq 2}|\partial u|^{p}}\|_{\fl{L}_r^{1} L_\omega^2}&\leq C\k({\k\|{\Gamma^{\leq 2}\partial u}\|_{\fl{L}_r^2 L_\omega^2}\k\|{\partial u}\|_{\fl{L}_r^2 L_\omega^\infty}+\k\|{\Gamma^{\leq 1}\partial u}\|_{\fl{L}_r^2 L_\omega^4}^2})\k\|{\partial u}\|_{\fl{L}_r^\infty L_\omega^\infty}^{p-2}\\
&\leq C\k\|{\Gamma^{\leq 2}\partial u}\|_{L_x^2}^2\k\|{\partial u}\|_{L_x^\infty }^{p-2} \\
&\leq CC_M^{p}\varepsilon^{p}\k<t>^{1-p/2}.
\end{aligned}
\end{equation*}
\par Since $p\geq q_c,$ $(1-q_c/2)q_*=-1,$ $T<\tilde T_\varepsilon$, we get
\begin{equation*}
\begin{aligned}
A(T)&\leq cC_M\varepsilon^{1/q_c}+cCC_M^{q_c}\varepsilon^{1/q_c-1+1}+cCC_M^{p}\varepsilon^{1/q_c-1+p}\k({\int_0^T \k<t>^{(1-p/2)q_*}\d t})^{1/q_*}\\
&\leq cC_M\varepsilon^{1/q_c}+cCC_M^{q_c}\varepsilon^{1/q_c-1+1}+cCC_M^{p}\varepsilon^{1/q_c-1+p}({\ln( T+2)})^{1/q_*}\\
&\leq C_M\varepsilon^{1/q_c}(c+cCC_M^{q_c-1}).
\end{aligned}
\end{equation*}
\part[Estimate of $B(t)$]
\par In this part, by applying the energy inequality, we obtain
\begin{equation*}
B(T)\leq C_M \varepsilon/4 +C\k\|{\Gamma^{\leq 2}|u|^{q_c}}\|_{L_t^1L_x^2 (t<T)}+C\k\|{\Gamma^{\leq 2}|\partial u|^{p}}\|_{L_t^1L_x^2(t<T)}.
\end{equation*}
\par For the second term, by \emph{H\"older's} inequality, Lemma \ref{l3.5} and Lemma \ref{l3.6}, we have
\begin{equation*}
\begin{aligned}
\k\|{\Gamma^{\leq 2}|u|^{q_c}}\|_{L_t^1L_x^2}&=\k\|{\Gamma^{\leq 2}|u|^{q_c}}\|_{L_t^1\fl{L}_r^2 L_\omega^2}\\
&\leq C\k\|{\k\|u\|_{\fl{L}_r^\infty L_\omega^\infty}^{q_c-2}\k({\k\|{\Gamma^{\leq 2}u}\|_{\fl{L}_r^{q_c} L_\omega^2}\k\|u\|_{\fl{L}_r^{q_c+1} L_\omega^\infty}+\k\|{\Gamma^{\leq 1}u}\|_{\fl{L}_r^{4} L_\omega^4}^2})}\|_{L_t^{1}}\\
&\leq C\k\|{\k\|{\Gamma^{\leq 1}u}\|_{L_x^{q_c}}^{q_c-2}\k\|{\Gamma^{\leq 2}u}\|_{\fl{L}_r^{q_c} L_\omega^2}^{2} \k<t>^{(-1/q_c)(q_c-2)}}\|_{L_t^{1}}\\
&\leq C\k\|{\k\|{\Gamma^{\leq 2}u}\|_{\fl{L}_r^{q_c} L_\omega^2}^{q_c}\k<t>^{q_c-4}}\|_{L_t^{1}}\\
&\leq C\k\|{\Gamma^{\leq 2}u}\|_{L_t^{q_*q_c}\fl{L}_r^{q_c} L_\omega^2}^{q_c}\k\|{\k<t>^{q_c-4}}\|_{L_t^{1+q_c}}\\
&\leq CC_M^{q_c}\varepsilon^{q_c},
\end{aligned}
\end{equation*}
and the same for the last term
\begin{equation*}
\begin{aligned}
\k\|{\Gamma^{\leq 2}|\partial u|^{p}}\|_{L_x^2}&\leq C\k\|{\Gamma^{\leq 2}\partial u}\|_{L_x^2}\|\partial u\|_{L_x^{\infty}}^{p-1}+C\k\|{\Gamma^{\leq 1}\partial u}\|_{L_x^4}^2\|\partial u\|_{L_x^{\infty}}^{p-2}\\
&\leq  C\k\|{\Gamma^{\leq 2}\partial u}\|_{L_x^2}^p\k({\k<t>^{-(1/2)(p-1)}+\k<t>^{-(1/2-1/4)2-(1/2)(p-2)}})\\
&\leq CC_M^{p}\varepsilon^{p}\k<t>^{(1-p)/2}.
\end{aligned}
\end{equation*}
\par By the definition of $\ \tilde T$ and $\varepsilon_0,$ we know
\begin{equation*}
\begin{aligned}
B(T)&\leq C_M\varepsilon \k({1/4+ C C_M^{q_c}\varepsilon^{q_c-1}+C C_M^{p-1}\varepsilon^{p-1}})\leq C_M\varepsilon \k({1/4+ C C_M^{q_c}\varepsilon_0^{q_c-1}}).
\end{aligned}
\end{equation*}
\part[The boundness of $A(t),B(t)$]
Now, by choosing $\varepsilon_0$ and the constant $c$ (in $\tilde T_\varepsilon$) small enough, we conclude $A(T)\leq C_M \varepsilon^{1/q_c}/2$ and $B(T) \leq C_M \varepsilon/2,$ which completes the proof.

\subsection{Proof of Claim \ref{c3.7}}\ \par

Before the discussion, set $h=\k\{{\partial_x^{\leq 1}f, g}\}$, by equation \cref{e1.1} we see
\begin{equation*}
\begin{aligned}
\k|{\k\{{\partial_x^{\leq 3} u, \partial_x^{\leq 2} \partial_tu}\}}|_{t=0}&\leq C\varepsilon\k|{\partial_x^{\leq 2} h}|, \\[10pt]
\k|{\partial_x^{\leq 1}\partial_t^2 u}|_{t=0}&=\k|{\partial_x^{\leq 1}\k({|\partial u|^p+|u|^{q_c}+\Delta u})}|_{t=0}\\
&\leq C\varepsilon\k|{\k\{{\partial_x^{\leq 1} h|h|^{p-1}, |h|^{q_c}, \partial_x^{\leq 2} h}\}}|, \\[10pt]
\k|{\partial_t^3u}|_{t=0}&\leq C\k|{\k\{{\partial \partial_t u\k|{\partial u}|^{p-1}, \partial_tu|u|^{q_c-1}, \partial_t\Delta u}\}}|_{t=0}\\
&\leq C\k|{\k\{{\k({|\partial u|^p+|u|^{q_c}+\Delta u})|\partial u|^{p-1}, \partial_x\partial_t u\k|{\partial u}|^{p-1}, \partial_tu|u|^{q_c-1}, \partial_t\Delta u}\}}|\\
&\leq C\varepsilon\k|{\k\{{ |h|^{2p-1}, |h|^{q_c+p-1}, \partial_x^{\leq 1} h|h|^{p-1}, |h|^{q_c}, \partial_x^{\leq 2} h  }\}}|.
\end{aligned}
\end{equation*}
Set $M=\Lambda+\Lambda^p+\Lambda^{q_c}+\Lambda^{2p-1}+\Lambda^{q_c+p-1},$ we want to show that all of the initial norms can be controlled by $C\varepsilon M,$ where $C$ does not depend on $\varepsilon$ and $M.$

\par We begin with $A(0)=\|\Gamma^{\leq 2} u(0, \cdot)\|_{\fl{L}_r^{q_c}L_\omega^3}(n=3)$, here we get
\begin{equation*}
\begin{aligned}
\|\Gamma^{\leq 2} u(0, \cdot)\|_{\fl{L}_r^{q_c}L_\omega^3}&\leq C\|\Gamma^{\leq 3} u(0, \cdot)\|_{L_x^{q_c}}\\
&\leq C\k\|{\k<x>^5\partial^{\leq 3} u(0, \cdot)}\|_{L_x^\infty}\k\|{\k<x>^{-2}}\|_{L_x^{q_c}}\\
&\leq C\varepsilon M.
\end{aligned}
\end{equation*}
\par Similarly, for $B(0)=\|\partial\Gamma^{\leq 2}u(0, \cdot)\|_{L_x^2}(n=2, 3),$  we have
\begin{equation*}
\begin{aligned}
\|\partial\Gamma^{\leq 2}u(0, \cdot)\|_{L_x^2}&\leq C\k\|{\k<x>^5\partial^{\leq 3} u(0, \cdot)}\|_{L_x^\infty}\k\|{\k<x>^{-3}}\|_{L_x^2}\leq C\varepsilon M.
\end{aligned} 
\end{equation*}
\par For $\|r^{1-2/q_c}\partial_\omega^{\leq 1}\Gamma^{\leq 2}u(0, \cdot)\|_{\fl{L}_r^{q_c}L_\omega^{3/2}}(n=3)$ which comes from right hand side (RHS) of equation \cref{e3.1}, we see
\begin{equation*}
\begin{aligned}
\|r^{1-2/q_c}\partial_\omega^{\leq 1}\Gamma^{\leq 2}u(0, \cdot)\|_{\fl{L}_r^{q_c}L_\omega^{3/2}}&\leq  C\|\k<x>^{4-2/q_c}\partial^{\leq 3}f\|_{L_x^{q_c}}\\
&\leq  C\k\|{\k<x>^{5}\partial^{\leq 3}u(0, \cdot)}\|_{L_x^\infty}\k\|{\k<x>^{-1-2/q_c}}\|_{L_x^{q_c}}\\
&\leq C\varepsilon M.
\end{aligned}
\end{equation*}
\par The other three terms from equation \cref{e3.1} can be controlled by similar arguments. At last, for the term $\k\|{\Gamma^{\leq 2}({u, \partial_t u})(0, \cdot)}\|_{X^b}(n=2)$ which comes from equation \cref{e3.2}, we have
\begin{equation*}
\begin{aligned}
\k\|{\Gamma^{\leq 2}({u, \partial_t u})(0, \cdot)}\|_{X^b}&\leq \k\|{\Gamma^{\leq 2}({u, \partial_t u})(0, \cdot)}\|_{\k({H^1, L^2})}.
\end{aligned}
\end{equation*}
\par Through the previous discussion, we finish the proof of the Claim \ref{c3.7}.

\section{Proof of Theorem \ref{t1.3}}
\par Before the proof, we consider a simple coordinate transform with $(u, u_t)|_{t=2}=(f, g)$, $(v, v_t)|_{t=2}=(\tilde f, \tilde g)$. Due to the property of symmetry, $(u, v)$ solve the equivalent $1$-D integral equations in $t\geq2$
\begin{align*}
u(t, r)&=\varepsilon u_{o}(t, r)+L|v|^q(t, r), \\
v(t, r)&=\varepsilon v_{o}(t, r)+L|\partial_t u|^p(t, r),
\end{align*}
where
\begin{align*}
u_o(t+2, r)&=\frac{1}{2r}\k({(r+t)f(r+t)+(r-t)f(r-t)+\int_{r-t}^{r+t}\rho g(\rho)\d\rho}), \\
v_o(t+2, r)&=\frac{1}{2r}\k({(r+t)\tilde f(r+t)+(r-t)\tilde f(r-t)+\int_{r-t}^{r+t}\rho \tilde g(\rho)\d\rho}), \\
LF(t, r)&=\frac{1}{2r}\int_0^t\int_{r-t+s}^{r+t-s}\rho F(s, \rho)\d\rho\d s.
\end{align*}
for $t\geq 0$. Here for convenience we consider $u|_{t<2}=v|_{t<2}=0.$

Here we denote $f(|x|)=f(-|x|)=f(x)$ and the rest is similar. Then the lower limits of the integrals (to $\rho$) may be replaced by $|r-t|$ or $|r - (t - s)|$ because of the symmetric assumption. To control the iteration procedure, we need to estimate some derivatives. With the notation $w=\partial_t u,$ we find that
\begin{align}
w(t, r)&=\varepsilon\partial_t u_o(t, r)+r^{-1}K_+|v|^q(t, r), \label{e4.1}\\
v(t, r)&=\varepsilon v_o(t, r)+L|w|^p(t, r), \label{e4.2}\\
\partial_r\{rv(t, r)\}&=\varepsilon \partial_r\big(rv_o(t, r)\big)+K_-|w|^p(t, r), \label{e4.3}
\end{align}
where
\begin{equation*}
K_{\pm}F(t, r)=\frac{1}{2}\int_0^t(r+t-s)F(s, r+t-s)\pm(r-t+s)F(s, r-t+s)\d s.
\end{equation*}

To control the norm of $(w, v),$ we set
\begin{equation*}
\|(w, v)\|=\|\omega_1 w\|_{L_{t, r}^\infty}+\|\omega_2 v\|_{L_{t, r}^\infty}+\|\omega_3 r\partial_r v\|_{L_{t, r}^\infty}
\end{equation*}
where weight functions $\omega_1, \omega_2, \omega_3$ are defined by
\begin{align*}
\omega_1(t, r)&=\begin{cases}\k<{r}>\k<{t-r}>^{\mu/p+q-2}&(r<t/2), \\\k<{t-r}>^{\mu/p}\k<{t+r}>^{q-1}&(r\geq t/2);\end{cases}\\
\omega_2(t, r)&=\begin{cases}\k<{r}>^{p-2}\k<{t+r}>^{3-p+\mu/pq}&(r<t/2), \\\k<{t-r}>^{\mu/pq}\k<{t+r}>&(r\geq t/2);\end{cases}\\
\omega_3(t, r)&=\k<{t-r}>^{\mu+pq-2p}
\end{align*}
for $t \geq 0$ and $r \geq 0,$ with a fixed $\mu<1$ which satisfies
\begin{equation}\label{e4.4}
-\mu-pq+2p\leq p-3-\mu/pq.
\end{equation}

Now we consider the system of integral equations \cref{e4.1}-\cref{e4.3} in the close subset of complete metric space
\begin{equation*}
\begin{aligned}
X_\varepsilon=\Big\{(w, v):&~w, v, r\partial_r v\in C\big([2,\infty)\times \R\big), \|(w, v)\|\leq C_1\varepsilon, \\
&\supp(w, v)\subset \{t-r\geq 1,t\geq 2\}\Big\},
\end{aligned}
\end{equation*}
where $C_1$ will be determined later.

\begin{lmm}\label{l4.1}
Suppose that \cref{e1.9} is satisfied, $(w, v)\in X_\varepsilon.$ Then we have
\begin{align}
&\|\omega_2 L|w|^p\|_{L_{t, r}^\infty}\leq C\|\omega_1w\|_{L_{t, r}^\infty}^p, \label{e4.5}\\
&\|\omega_1r^{-1}K_+|v|^q\|_{L_{t, r}^\infty}\leq C\|\omega_2 v\|_{L_{t, r}^\infty}^q+C\|\omega_2 v\|_{L_{t, r}^\infty}^{q-1}\|\omega_3r\partial_rv\|_{L_{t, r}^\infty}, \label{e4.6}\\
&\|\omega_3 K_-|w|^p\|_{L_{t, r}^\infty}\leq C\|\omega_1 w\|_{L_{t, r}^\infty}^p.\label{e4.7}
\end{align}
\end{lmm}
\begin{lmm}
Suppose that \cref{e1.9} is satisfied, $(w, v), (\bar w, \bar v)\in X_\varepsilon.$ Set
\begin{equation*}
\begin{aligned}
\tilde\omega_1&=\begin{cases}r\k<{t-r}>^{\mu/p+q-2}&(r<t/2), \\\k<{t-r}>^{\mu/p}\k<{t+r}>^{q-1}&(r\geq t/2);\end{cases}\\
\tilde\omega_2&=\begin{cases}r^{p-2}\k<{t+r}>^{3-p+\mu/pq}&(r<t/2), \\\k<{t-r}>^{\mu/pq}\k<{t+r}>&(r\geq t/2).\end{cases}
\end{aligned}
\end{equation*}
Then we have
\begin{align}
&\|\tilde\omega_2 L\k({|w|^p-|\bar w|^p})\|_{L_{t, r}^\infty}\leq C\|\tilde\omega_1(w, \bar w)\|_{L_{t, r}^\infty}^{p-1}\|\tilde\omega_1(w-\bar w)\|_{L_{t, r}^\infty}, \label{e4.8}\\
&\|\tilde\omega_1 r^{-1}K_+\k({|v|^q-|\bar v|^q})\|_{L_{t, r}^\infty}\leq C\|\tilde\omega_2(v, \bar v)\|_{L_{t, r}^\infty}^{q-1}\|\tilde\omega_2(v-\bar v)\|_{L_{t, r}^\infty}. \label{e4.9}
\end{align}
\end{lmm}

Here, we apply the fixed point theorem with mapping
$$P: (w, v)\mapsto(Pw, Pv)\dyw(\varepsilon \partial_t u_o+r^{-1}K_+|v|^q, \varepsilon v_o+L|w|^p). $$
\par Firstly we check that $P$ is well defined in $X_\varepsilon\rightarrow X_\varepsilon.$ By expression \cref{e4.1}-\cref{e4.3}, it is obvious that $ \supp(Pw, Pv)\subset \{t-r\geq 1,t\geq 2\}$ and $Pw, Pv, r\partial_rPv\in C(\R^+\times \R) .$ To estimate $\|(Pw, Pv)\|,$ we begin with
\begin{equation*}
\begin{aligned}
\|\omega_1Pw\|_{L_{t, r}^\infty}&\leq \varepsilon\|\omega_1\partial_t u_o\|_{L_{t, r}^\infty}+\|\omega_1r^{-1}K_+|v|^q\|_{L_{t, r}^\infty}.
\end{aligned}
\end{equation*}
\par Since $f\in C^2,$ $g\in C^1,$ and $\supp u_o \subset \{(t, r):3\geq t-r\geq 1\}$ where $\omega_1\lesssim\k<t>,$ we see
$$\|\omega_1\partial_t u_o\|_{L_{t, r}^\infty}\leq C_{f,g}.$$
Moreover, by \cref{e4.6}, and $(w, v)\in X_\varepsilon, $ we conclude
\begin{equation*}
\|\omega_1Pw\|_{L_{t, r}^\infty}\leq C_{f, g}\varepsilon+C\varepsilon^q.
\end{equation*}

The estimates for the remaining terms are similar, noticing $r\partial_r Pv=\partial_r(rPv)-Pv$ and $\omega_3\lesssim\omega_2,$ by \cref{e4.5} and \cref{e4.7} we finally have
\begin{equation*}
\|(w, v)\|\leq C_{f, g, \tilde f, \tilde g}\varepsilon+C\varepsilon^p+C\varepsilon^q\leq C_1\varepsilon
\end{equation*}
for $C_1\geq 2C_{f, g, \tilde f, \tilde g}$ and $\varepsilon$ small enough.
\par Similarly, by \cref{e4.8} and \cref{e4.9} we have $P$ is contraction in a weaker sense. However it is enough to obtain the fixed point $(u, v)$ which solves \cref{e4.1}-\cref{e4.3}. So we complete the proof.

\subsection{Proof of \cref{e4.5} and \cref{e4.8}}\
\par First we prove \cref{e4.5}. Let $r > 0.$ Considering $D=\{(s, \rho):t-r\leq s+\rho\leq t+r, 1\leq s-\rho\leq t-r\}$ is the influence domain of $(t, r)$ intersect with $\{(s, \rho):s-\rho\geq 1\}.$ Set $D_1=D\cap \{(s, \rho):\rho<s/2\},$ $D_2=D\cap \{(s, \rho):\rho\geq s/2\},$ then we find
\begin{equation}\label{e4.10}
\begin{aligned}
|\omega_2 L|w|^p|&\leq \omega_2\frac{C}{r}\int_{D}\rho\, \omega_1^{-p}(s, \rho)\|\omega_1 w\|_{L_{t, r}^\infty}^p\d\rho\d s\\
&= C\omega_2\|\omega_1 w\|_{L_{t, r}^\infty}^p\frac{1}{r}\k({\int_{D_1}+\int_{D_2}})\rho\, \omega_1^{-p}(s, \rho)\d\rho\d s.
\end{aligned}
\end{equation}
\part[$(s,\rho)\in D_1$]
Here $\k<{s}>\approx\k<{s-\rho}>\approx\k<{s+\rho}>,$ take $\tau=s+\rho$, $\sigma=s-\rho,$ by \cref{e4.4} we obtain
\begin{equation}\label{e4.11}
\begin{aligned}
\frac{1}{r}\int_{D_1}\rho\, \omega_1^{-p}(s, \rho)\d\rho\d s&=\frac{1}{r}\int_{D_1}\rho\, \k<{\rho}>^{-p}\k<{s-\rho}>^{-\mu-pq+2p}\d\rho\d s\\
&\leq \frac{C}{r}\int_{t-r}^{t+r}\int_{(t-r)/3}^{t-r}\k<{\tau-\sigma}>^{1-p}\k<{\tau}>^{-\mu-pq+2p}\d\sigma\d\tau\\
&\leq \frac{C}{r}\int_{t-r}^{t+r}\k<{\tau-t+r}>^{2-p}\k<{\tau}>^{p-3-\mu/pq}\d\tau.
\end{aligned}
\end{equation}
\subpart[$r<t/2$]
Here $t-r\approx t\approx t+r,$ we conclude
\begin{equation}\label{e4.12}
\begin{aligned}
RHS~of~\cref{e4.11}&\leq C\k<{t+r}>^{p-3-\mu/pq}r^{-1}\k({(r+1)^{3-p}-1^{3-p}})\\
&\leq C\k<{t+r}>^{p-3-\mu/pq}\k<{r}>^{2-p}\\
&= C\omega_2^{-1}.
\end{aligned}
\end{equation}
\subpart[$r\geq t/2$]
Here $r\approx t\approx t+r,$ similarly we have
\begin{equation}
\begin{aligned}
RHS~of~\cref{e4.11}&\leq C\k<{t}>^{-1}\bigg(\k<{t-r}>^{p-3-\mu/pq}\int_{t-r}^{2(t-r)}\k<{\tau-t+r}>^{2-p}\d\tau\\
&\hspace{50pt}+\int_{2(t-r)}^{t+r}\k<{\tau}>^{-1-\mu/pq}\d\tau\bigg)\\
&\leq C\k<t>^{-1}\k<{t-r}>^{-\mu/pq}\\
&\leq C\omega_2^{-1}.
\end{aligned}
\end{equation}
\part[$(s,\rho)\in D_2$]
Here $\k<{\rho}>\approx\k<{s}>\approx\k<{s+\rho}>,$ take $\tau=s+\rho$, $\sigma=s-\rho, $ then we conclude
\begin{equation}\label{e4.14}
\begin{aligned}
\frac{1}{r}\int_{D_2}\rho\, \omega_1^{-p}(s, \rho)\d\rho\d s&=\frac{1}{r}\int_{D_2}\rho\, \k<{s-\rho}>^{-\mu}\k<{s+\rho}>^{-p(q-1)}\d\rho\d s\\
&\leq \frac{C}{r}\int_{t-r}^{t+r}\k<{\tau}>^{1-p(q-1)}\d\tau\int_{1}^{t-r}\k<{\sigma}>^{-\mu}\d\sigma.\\
\end{aligned}
\end{equation}
\subpart[$r<t/2$]
Here we have
\begin{equation}
\begin{aligned}
RHS~of~\cref{e4.14}\leq C\k<{t-r}>^{1-p(q-1)}\k<{t-r}>^{1-\mu}\leq C\omega_2^{-1}.
\end{aligned}
\end{equation}
\subpart[$r\geq t/2$]
Here we have
\begin{equation}\label{e4.16}
\begin{aligned}
RHS~of~\cref{e4.14}\leq C\k<t>^{-1}\k<{t-r}>^{2-p(q-1)}\k<{t-r}>^{1-\mu}\leq C\omega_2^{-1}.
\end{aligned}
\end{equation}
\par Thus \cref{e4.10}-\cref{e4.16} give $|\omega_2 L|w|^p|\leq C\|\omega_1 w\|_{L_{t, r}^\infty}^p$ which completes the proof of \cref{e4.5}. The proof of \cref{e4.8} is similar, since that $\big|{|a|^q-|b|^q}\big|\lesssim\k({|a|^{q-1}+|b|^{q-1}})|a-b|$ for $q>1.$
\subsection{Proof of \cref{e4.6} and \cref{e4.9}}\
\par For \cref{e4.6}, we divide the proof into two main cases: $r\geq 1/4$ and $r<1/4.$

\part[$r\geq 1/4$]
For this situation, we see $r\approx \k<r>.$ Similar to the proof of \cref{e4.5} we have
\begin{equation}
\begin{aligned}
\k|{\omega_1 r^{-1}K_+|v|^q}|&\leq\|\omega_2 v\|_{L_{t, r}^\infty}^q \omega_1r^{-1}\int_0^t \frac{(r+t-s)}{\omega_2(s, r+t-s)^q}+\frac{|r-t+s|}{\omega_2(s, |r-t+s|)^q}\d s\\
&\equiv \|\omega_2 v\|_{L_{t, r}^\infty}^q \omega_1r^{-1}\int_0^t \Roma1+\Roma2\d s
\end{aligned}
\end{equation}
\subpart[$r<t/2$. Estimates about $\Roma1$]
In this part, we obtain $\omega_1r^{-1} \approx \k<{t-r}>^{\mu/p+q-2}.$

\subsubpart[$0\leq s< 2(r+t)/3$]
This means $r+t-s> s/2,$ then we conclude
\begin{equation}\label{e4.18}
\begin{aligned}
\int_0^{2(t+r)/3} \Roma1\d s&=\int_0^{2(t+r)/3} (r+t-s)\k<{2s-r-t}>^{-\mu/p}\k<{r+t}>^{-q}\d s\\
&\leq C \k<{t+r}>^{1-q}  \int_0^{2(t+r)/3}\k<{2s-r-t}>^{-\mu/p}\d s \\
&\leq C\k<{t-r}>^{2-q-\mu/p}\\
&\leq Cr\omega_1^{-1}.
\end{aligned}
\end{equation}
\subsubpart[$2(r+t)/3\leq s\leq t$]
This means $r+t-s\leq s/2,$ then similarly we have
\begin{equation}\label{e4.19}
\begin{aligned}
\int_{2(t+r)/3}^t \Roma1\d s&=\int_{2(t+r)/3}^t (r+t-s)\k<{r+t-s}>^{-q(p-2)}\k<{r+t}>^{-3q+pq-\mu/p}\d s\\
&\leq C \k<{t+r}>^{2-q(p-2)}\k<{r+t}>^{-3q+pq-\mu/p}\\
&\leq Cr\omega_1^{-1}.
\end{aligned}
\end{equation}
\subpart[$r\geq t/2$. Estimates about $\Roma1$]
Here we always have $r+t-s\geq s/2$. Then it is similar to \cref{e4.18}.

\subpart[$r\geq t/2$. Estimates about $\Roma2$]

\subsubpart[$0\leq s<t-r$]
\par Here  $|r-t+s|=t-s-r.$ It is similar to \cref{e4.18}-\cref{e4.19}.

\subsubpart[$t-r\leq s<2(t-r)$]
Here $r-t+s<s/2,$ then we have
\begin{equation}
\begin{aligned}
\int_{t-r}^{2(t-r)} \Roma2\d s&=\int_{t-r}^{2(t-r)} (r-t+s)\k<{r-t+s}>^{-q(p-2)}\k<{r-t+2s}>^{-3q+pq-\mu/p}\d s\\
&\leq C \k<{t-r}>^{2-q(p-2)}\k<{t-r}>^{-3q+pq-\mu/p}\\
&\leq Cr\omega_1^{-1}.
\end{aligned}
\end{equation}
\subsubpart[$2(t-r)\leq s<t$]
\par Here $r-t+s\geq s/2,$ then we also have
\begin{equation}
\begin{aligned}
\int_{2(t-r)}^{t} \Roma2\d s&=\int_{2(t-r)}^{t} (r-t+s)\k<{t-r}>^{-\mu/p}\k<{r-t+2s}>^{-q}\d s\\
&\leq C \k<{t-r}>^{-\mu/p}\int_{2(t-r)}^{t}\k<{r-t+2s}>^{1-q}\d s\\
&\leq C\k<{t-r}>^{-\mu/p}\k<{t+r}>^{2-q}\\
&\leq Cr\omega_1^{-1}.
\end{aligned}
\end{equation}
\subpart[$r<t/2$. Estimates about $\Roma2$]

\subsubpart[$0\leq s<t-r$]
It is similar to \cref{e4.18}-\cref{e4.19}.

\subsubpart[$t-r\leq s\leq t$]
Here we have
\begin{equation}\label{e4.22}
\begin{aligned}
\int_{t-r}^{t} \Roma2\d s&=\int_{t-r}^{t} (r-t+s)\k<{r-t+s}>^{-q(p-2)}\k<{r-t+2s}>^{-3q+pq-\mu/p}\d s\\
&\leq C \k<{r}>^{2-q(p-2)}\k<{t+r}>^{-3q+pq-\mu/p}\\
&\leq Cr\omega_1^{-1}.
\end{aligned}
\end{equation}
\part[$r<1/4$]
In this part $\k<r>\approx 1$. For the convenience of proof, we set $r^+=r+t-s$, $r^-=r-t+s.$ Then we get
\begin{equation}\label{e4.23}
\begin{aligned}
|\omega_1r^{-1}K_+|v|^q|\leq C \omega_1r^{-1}\int_0^t \Big|{r^+|v|^q(s, r^+)+r^-|v|^q(s, |r^-|)}\Big|\d s
\end{aligned}
\end{equation}
where
\begin{equation*}
\begin{aligned}
\Big|{r^+|v|^q(s, r^+)+r^-|v|^q(s, |r^-|)}\Big|\leq&\Big|{r^+|v|^q(s, r^+)+r^-|v|^q(s, r^+)}\Big|\\
&+\Big|{r^-|v|^q(s, r^+)-r^-|v|^q(s, |r^-|)}\Big|
\end{aligned}
\end{equation*}
then
\begin{equation*}
\begin{aligned}
RHS~of~\cref{e4.23}&\leq C \omega_1\k({\int_0^t|v|^q(s, r^+)\d s+r^{-1}\int_0^t|r^-| \Big||v|^q(s, r^+)-|v|^q(s, |r^-|)\Big|\d s})\\
&\equiv C \omega_1\Roma1+C \omega_1\Roma2.
\end{aligned}
\end{equation*}
\subpart[Estimates about $\Roma1$]
It is similar to \cref{e4.18}-\cref{e4.19}.

\subpart[Estimates about $\Roma2$]
For this part, we should consider $\partial_r v.$ Since $r^+-|r^-|\leq 2r,$ we obtain
$$r^{-1}\Big||v|^q(s, r^+)-|v|^q(s, |r^-|)\Big|\leq C |\partial_r v(s, r^*)||v|^{q-1}(s, r^*)$$
for $|r^-|<r^*(s)<r^+.$ Then we have
\begin{equation*}
\begin{aligned}
\Roma2&\leq C \int_0^t|r^-||\partial_r v(s, r^*)||v|^{q-1}(s, r^*)\d s\\
&\leq C\|\omega_2 v\|_{L_{t, r}^\infty}^{q-1}\|\omega_3r\partial_rv\|_{L_{t, r}^\infty} \int_0^t\k({\omega_2^{-(q-1)}\omega_3^{-1}})(s, r^*)\d s\\
&\equiv C\|\omega_2 v\|_{L_{t, r}^\infty}^{q-1}\|\omega_3r\partial_rv\|_{L_{t, r}^\infty} \int_0^t\Roma3\d s
\end{aligned}
\end{equation*}
Here we notice $\k<t>\lesssim\k<{|r^-|+s+r}>\lesssim\k<{r^*+s}>.$

\subsubpart[$r^*<s/2$]
Here $ \k<{s\pm r^*}>\gtrsim \k<{r^*}>\gtrsim \k<{t-s}>,$ by \cref{e4.4}, we see
\begin{equation*}
\begin{aligned}
\omega_1\int_{2r^*<s<t}\Roma3\d s&\leq C\int_{2r^*<s<t} \k<{r^*}>^{(1-q)(p-2)}\\
&\hspace{50pt}\times\k<{s-r^*}>^{(1/p-1)(\mu+pq-2p)+(1-q)(3-p+\mu/pq)}\d s\\
&\leq C\int_{2r^*<s<t} \k<{r^*}>^{-1-\mu-pq+2p+\mu/pq}\d s\\
&\leq C\int_{2r^*<s<t} \k<{t-s}>^{p-4}\d s\\
&\leq C.
\end{aligned}
\end{equation*}
\subsubpart[$r^*\geq s/2$]
Here $\k<{s+r^*}>\gtrsim \k<{s-r^*}>\gtrsim \k<{2s-t}>,$ similarly we have
\begin{equation*}
\begin{aligned}
\omega_1\int_{s<2r^*\wedge t}\Roma3\d s\leq C\int_{s<2r^*\wedge t}\k<{2s-t}>^{p-4}\d s\leq C.
\end{aligned}
\end{equation*}
In summary, we complete the proof. To verify \cref{e4.9}, without distinguishing whether $r>1/4$ or not we have a proof just like \cref{e4.10}-\cref{e4.22}.
\subsection{Proof of \cref{e4.7}}\
\par We follow the same process as before
\begin{equation*}
\begin{aligned}
\k|{\omega_3 K_-|w|^p}|\leq\|\omega_1 v\|_{L_{t, r}^\infty}^p \omega_3\int_0^t \frac{r+t-s}{\omega_1(s, r+t-s)^p}+\frac{|r-t+s|}{\omega_1(s, |r-t+s|)^p}\d s.
\end{aligned}
\end{equation*}

Both part in the integration is similar to the last proof. Here we only show the proof of $r\geq t/2$, $(t-r)\leq s$ for $|r-t+s|$ part.

\part[$t-r\leq s<2(t-r)$]
Here $r-t+s<s/2$, then we see
\begin{equation*}
\begin{aligned}
\int_{t-r}^{2(t-r)} \frac{r-t+s}{\omega_1(s, r-t+s)^p}\d s&=\int_{t-r}^{2(t-r)} (r-t+s)\k<{r-t+s}>^{-p}\k<{r-t+2s}>^{-\mu-pq+2p}\d s\\
&\leq C \k<{t-r}>^{-\mu-pq+2p}\\
&\leq C\omega_3^{-1}.
\end{aligned}
\end{equation*}
\part[$2(t-r)\leq s\leq t$]
Here $r-t+s\geq s/2$, then we know
\begin{equation*}
\begin{aligned}
\int_{2(t-r)}^{t} \frac{r-t+s}{\omega_1(s, r-t+s)^p}\d s&=\int_{2(t-r)}^{t} (r-t+s)\k<{t-r}>^{-\mu}\k<{r-t+2s}>^{-pq+p}\d s\\
&\leq C \k<{t-r}>^{-\mu}\int_{2(t-r)}^{t}\k<{r-t+2s}>^{1-pq+p}\d s\\
&\leq C\k<{t-r}>^{-\mu}\k<{t-r}>^{2-pq+p}\\
&\leq C\omega_3^{-1},
\end{aligned}
\end{equation*}
which complete the proof.

\section{Proof of Theorem \ref{t1.4}}\label{s5}
Following \cite{MR2195336}, we introduce two positive test functions
\begin{equation*}
\phi(x)=\int_{S^{n-1}}e^{x\cdot \omega}\d\omega, \qquad \psi(t, x)=e^{-t}\phi(x),
\end{equation*}
where it is understood that $\phi(x)=e^{x}+e^{-x}$ when $n=1$.
It is well known that we have
\begin{prpstn}\label{l5.1}
For $\phi, \psi$ defined above, then we have
\begin{equation*}
\begin{aligned}
&\psi(t, x)\leq C \k<{r}>^{-(n-1)/2}e^{r-t}, \\
&\|\psi(t, x)\|_{L_x^s(r<t+1)}\leq C \k<{t}>^{(n-1)(1/s-1/2)},
\end{aligned}
\end{equation*}
for any $s\in [1,\infty]$.
\end{prpstn}

We will prove Theorem \ref{t1.4} by contradiction. Suppose there are global weak solutions of \cref{e1.5} with $\supp (u,v)\subset\{(t, x):|x|\leq t+1\}$ and
\eqref{eq-1.11}.
As in \cite{MR2729246}, we define
\begin{equation*}
F(t)=\int_{\R^n} u(t, x)\psi(t, x)\d x, \qquad G(t)=\int_{\R^n} v(t, x)\psi(t, x)\d x.
\end{equation*}
Since $u, v, \partial_t u, \partial_t v\in C\big(\R_+;L^1(\R^n)\big),$ we know that $F, G\in C^1(\R_+)$
and
$$F'(t)=\int_{\R^n} u_t(t, x)\psi(t, x)+u(t, x)\psi_t(t, x)\d x
=\int_{\R^n} u_t(t, x)\psi(t, x)\d x-F(t)\ .$$
By the definition of weak solutions to \cref{e1.5}, with $\psi$ as test function, combined with the assumption $v\in C (\R_+;L^q)$,  we know that
$F\in C^2(\R_+)$ and
$$F''(t)=\int_{\R^n} |v|^q(t, x)\psi(t, x)\d x-2F'(t)\ .$$
Similarly, 
as $\partial_tu\in C (\R_+;L^p)$,
 we conclude $G\in C^2(\R_+)$ and
$$G''(t)=\int_{\R^n} |\partial_t u(t, x)|^p\psi(t, x)\d x-2G'(t)\ .
$$

Since $\supp (u,v)\subset\{(t, x):|x|\leq t+1\}$,
using \emph{H\"older's} inequality and Proposition \ref{l5.1}, we know that
\begin{equation}\label{e5.1}
F''(t)+2F'(t)=\int_{\R^n} |v(t, x)|^q\psi(t, x)\d x\geq C \k<{t}>^{-\frac{n-1}2(q-1)}|G(t)|^q,
\end{equation}
$$G''(t)+2G'(t)=\int_{\R^n} |\partial_t u|^p\psi \d x
\geq C \k<{t}>^{-\frac{n-1}2(p-1)}|F'(t)+F(t)|^p.$$

Notice that we have $F''(t)+2F'(t)\geq 0$ by \cref{e5.1}. By the assumption of data we have $F'(0)=\int_{\R^n} (g-f)\phi\d x\geq 0$, then it is easy to conclude that $F'(t)\geq 0$, for any $t\ge 0$. Moreover, since $F(0)=\int_{\R^n} f\phi\d x\geq 0, $ we have $F(t)\geq 0$ for all $t\geq 0.$ Then it is obvious that
\begin{equation*}
|F'(t)+F(t)|\geq \frac{1}{2}|F'(t)+2F(t)|.
\end{equation*}

By a similar argument,  since $\tilde f$ does not vanish identically, we have $G'(t)\geq 0$, $G(t)\geq C\varepsilon$. Set $H(t)=F'(t)+2F(t),$ then we have $H\in C^1(\R_+)$,
$G\in C^2(\R_+)$
 and
\begin{equation}\label{e5.2}
\begin{cases}
H'(t)\geq C
\k<{t}>^{-\frac{n-1}2(q-1)} G(t)^q, \\
H(t)\geq 0, \qquad H'(t)\geq 0; \\
G''(t)+2G'(t)\geq C \k<{t}>^{-\frac{n-1}2(p-1)
}H(t)^p, \\
G(t)\geq C\varepsilon, \qquad G'(t)\geq 0.
\end{cases}
\end{equation}
\begin{lmm}\label{l5.2}
For system \cref{e5.2}, assume \cref{e1.10} is satisfied, then for any $M>0,$ there exist constants $A, T$, which may depend on $\varepsilon$,  such that we have 
\begin{equation*}
G(t)\geq A \k<{t}>^M,
\end{equation*}
for any $t\geq  T$.
\end{lmm}
\begin{lmm}\label{l5.3}
Under the same assumption of Lemma \ref{l5.2},
there are no $H\in C^1(\R_+)$,
$G\in C^2(\R_+)$ satisfying
 the system \cref{e5.2}.
\end{lmm}
This gives the desired contradiction, which completes the proof.

\subsection{Proof of Lemma \ref{l5.2}}~

At first, we see that
\begin{equation*}
G(t)\geq A_0 \k<t>^{\alpha_0},  \forall\, t\ge T_{0}, 
\end{equation*}
with
$A_0=C\varepsilon$,  $\alpha_0=0$,  $T_{0}=1$.
We claim that for any $m\ge 0$, there exists $A_m>0$ such that we have
\begin{equation}\label{eq-indk}
G(t)\geq A_m \k<t>^{\alpha_m},\ \forall\, t\ge T_{m},
\end{equation}
where
\beeq\label{eq-indk2}\alpha_{m+1}=
pq\alpha_m+p+1-\frac{n-1}2(pq-1)
,\ T_{m}=8^m T_{0}\ .\eneq

With help of the claim, we see 
from \cref{e1.10} 
that $$\alpha_1=p+1-\frac{n-1}2(pq-1)>0,
\al_m> (pq)^{m-1} \al_1, \forall m\ge 1\ ,
$$ which gives us the desired property
$\lim_{m\to \infty}\alpha_m=\infty$ and completes the proof.

It remains to prove the claim, for which we use induction.
Assuming that 
for some $k\ge 0$,
 we have \eqref{eq-indk} for any $m\le k$, then by \cref{e5.2} we obtain
$$H'(t) \geq C A_k^q \k<t>^{
-\frac{n-1}2(q-1)
+q\alpha_k},\ \forall\, t\ge T_{k}\ .$$
As
$\frac{n-1}2(q-1)<1$ by 
\cref{e1.10},
we have
$-\frac{n-1}2(q-1)
+q\alpha_k>-1$ and
\beeq \label{eq-5.3}H(t)\geq H(T_{k})+C A_k^q \int_{T_{k}}^t \k<s>^{-\frac{n-1}2(q-1)+q\alpha_k}\d s
\geq \tilde C  \k<t>^{-\frac{n-1}2(q-1)+1+q\alpha_k}
\eneq
for any $t\ge 2T_{k}$. 
Plugging the lower bound \eqref{eq-5.3} to the ordinary differential inequality for $G$ in \cref{e5.2}, we get
$$G''(t)+2G'(t)\geq C \tilde C^{p} \k<{t}>^{
-\frac{n-1}2(p-1)+p(-\frac{n-1}2(q-1)+1+q\alpha_k)}$$ for all $t\ge 2T_{k}$,
which, by using the multiplier $e^{2t}$ and integration, yields
$$G'(t)e^{2t}\geq C  \k<{t}>^{-\frac{n-1}2(pq-1)+p+pq\alpha_k}e^{2t}$$ for any $t\ge 4 T_{k}$.
Recall that
$p-\frac{n-1}{2}(pq-1)>-1$
by \cref{e1.10},
we have
$$G(t)\geq 
G(4T_k)+C\int_{4T_k}^t \k<{s}>^{-\frac{n-1}2(pq-1)+p+pq\alpha_k}ds
\ge
\tilde C\k<{t}>^{
-\frac{n-1}2(pq-1)+p+1+pq\alpha_k}$$
for all  $t\ge 8T_{k}$,
which gives us \eqref{eq-indk} with $m=k+1$ and completes the proof of the claim.

\subsection{Proof of Lemma \ref{l5.3}}~

 First of all we want to simplify the system \cref{e5.2}. For any $t\ge 0$, we have
\begin{equation*}
\begin{aligned}
H(t)\big(G'(t)+2G(t)\big)&=\int_{0}^t H'(s)(G'(s)+2G(s))\d s\\
&\hspace{10pt}+\int_{0}^t H(s)(G''(s)+2G'(s))\d s+H(0)(G'(0)+2G(0))\\
&\geq C\int_{0}^t \k<{s}>^{-\frac{n-1}2(q-1)}G(s)^q(G'(s)+2G(s))\d s\\
&\geq C\k<{t}>^{-\frac{n-1}2(q-1)}\int_{0}^t G(s)^qG'(s)\d s\\
&\geq \tilde C\k<{t}>^{-\frac{n-1}2(q-1)}\k({ G^{q+1}(t)-G^{q+1}(0)  }).
\end{aligned}
\end{equation*}
Similarly, we have
$$\big(G'(t)+2G(t)\big)^p\big(G''(t)+2G'(t)\big)\geq C \k<{t}>^{
-\frac{n-1}2(p-1)}\Big(H(t)\big(G'(t)+2G(t)\big)\Big)^p\ .$$
Gluing together the above two inequalities, we obtain
$$\frac{d}{dt}(G'(t)+2G(t)\big)^{p+1} \geq C \k<{t}>^{-\frac{n-1}2(pq-1)} \k({ G^{q+1}(t)-G^{q+1}(0)  })^p\ ,$$
which gives us
\beeq\label{eq-s5.3} (G'(t)+2G(t))^{p+1}\geq  C\k<{t}>^{-\frac{n-1}2(pq-1)} \int_{0}^t \k({ G^{q+1}(s)-G^{q+1}(0)  })^p\d s
\eneq
for any $t>0$. Here,
since $G, G'>0$ for $t>0$ and $G'+2G, G$ are monotonically increasing to infinity,
 we have
\begin{equation*}
\begin{aligned}
\int_{0}^t \k({ G^{q+1}(s)-G^{q+1}(0)  })^p\d s &\geq C\int_{0}^t \k({ G^{q+1}(s)-G^{q+1}(0)  })^p\frac{\k({G(s)^{q+1}})'}{\big(G'(s)+2G(s)\big)G(s)^q}\d s\\
&\geq C\frac{\int_{0}^t \k({ G^{q+1}(s)-G^{q+1}(0)  })^p{\k({G(s)^{q+1}})'}\d s}{\big(G'(t)+2G(t)\big)G(t)^{q}}\\
&= \frac{C\k({ G^{q+1}(t)-G^{q+1}(0)  })^{p+1}}{(p+1)\big(G'(t)+2G(t)\big)G(t)^{q}}\ .
\end{aligned}
\end{equation*}
Plugging it into \eqref{eq-s5.3}, we obtain
$$\big(G'(t)+2G(t)\big)^{p+2}\geq  C
\k<{t}>^{-\frac{n-1}2(pq-1)}
\frac{\k({ G^{q+1}(t)-G^{q+1}(0)  })^{p+1}}{G(t)^{q}}\ .
$$
Choosing a $\tilde T_1>0$ such that $G(\tilde T_1)^{q+1}\geq 2G(0)^{q+1},$ we have for any $t>\tilde T_1$
$$\big(G'(t)+2G(t)\big)^{p+2}\geq   C
\k<{t}>^{-\frac{n-1}2(pq-1)}
G(t)^{pq+p+1}\ .$$
As $pq+p+1>p+2$,
there exists a $\delta>0$ such that
$pq+p+1>(p+2)(1+\de)+\de$. Let $M=\frac{n-1}2(pq-1)/\de$ in Lemma \ref{l5.2}, we see that there are $\tilde T_2\ge \tilde T_1$ and $C>0$ such that
$$G'(t)+2G(t)\geq   C
G(t)^{1+\de}\ ,$$
for any $t\ge \tilde T_2$. Moreover we can always take a $\tilde T_3\ge\tilde T_2$ such that $C G(t)^{1+\de}>4G(t)$ for all $t\ge \tilde T_3$, and so we arrived 
at the desired
ordinary differential inequality
 $$G'(t)\geq   \frac{C}{2}
G(t)^{1+\de}\ ,\ 
G(\tilde T_3)>0
,$$
for any $t\ge \tilde T_3$,
which blows up in finite time. This completes the proof.

\subsection*{Acknowledgment}
The authors would like to thank the anonymous referee for the careful reading and valuable comments.
This work was supported by NSFC 11671353. The third author was supported in part by National Support Program for Young Top-Notch Talents.

\end{document}